\documentclass[english]{amsart} 

\usepackage{amssymb}
\usepackage{enumerate}

\newtheorem{theorem}{Theorem}
\newtheorem{proposition}{Proposition}[section]
\newtheorem{lemma}[proposition]{Lemma}
\newtheorem{corollary}[proposition]{Corollary}

\theoremstyle{definition}
\newtheorem{definition}[proposition]{Definition}

\newtheorem{example}[proposition]{Example}

\newcommand{\p}{\mathbb{P}}
\renewcommand{\k}{\mathrm{k}}
\newcommand{\Aut}{\mathrm{Aut}}

\newcommand{\A}{\mathbb{A}}

\newcommand{\C}{\mathbb{C}}
\newcommand{\car}{\mathrm{char}}
\newcommand\dashmapsto{\mapstochar\dashrightarrow}

\DeclareMathOperator{\Bir}{Bir}

\title{Algebraic elements of the Cremona groups}
\author{J\'er\'emy Blanc}
\address{Mathematisches Institut \\ 
Universit\"at Basel \\
Spiegelgasse 1 \\
CH-4051 Basel \\
Switzerland}
\email{jeremy.blanc@unibas.ch}

\thanks{The author acknowledges support by the Swiss National Science Foundation Grant  "Birational Geometry" PP00P2\_153026.}

\date{\today}

\begin{document}
\begin{abstract}This article studies algebraic elements of the Cremona group. In particular, we show that the set of all these elements is a countable union of closed subsets but it is not closed.\end{abstract}
\maketitle

\section{Introduction}
In the sequel, the ground field~$\k$ will be a fixed algebraically closed field. The Cremona group of rank~$n$ is the group~$\Bir(\p^n)$ of birational transformations of the projective space~$\p^n$. 
There is a natural topology on it, called the \emph{Zariski topology} (see Section~\ref{Sec:Families} for a precise definition). 

 An element~$\varphi\in \Bir(\p^n)$ is said to be \emph{algebraic} if it is contained in an algebraic subgroup~$G$ of~$\Bir(\p^n)$. This is equivalent to the fact that the sequence~$\{\deg(\varphi^m)\}_{m\in \mathbb{N}}$ is bounded (Corollary~\ref{Cor:AlgSequenceBounded}). We can also observe that~$G$ is in this case an affine algebraic group \cite[Remark 2.21]{BFannals}, so there exists a Jordan decomposition~$\varphi=\varphi_s \varphi_u$, where~$\varphi_s,\varphi_u\in G$ are semi-simple and unipotent respectively. As observed in \cite[\S9.1]{Popov}, this decomposition does not depend on the choice of~$G$, so there is a natural notion of semi-simple and unipotent elements of~$\Bir(\p^n)$. In fact, the group~$G$ could even be chosen to be the commutative algebraic subgroup~$\overline{\{\varphi^i\mid i\in \mathbb{Z}\}}$ of~$\Bir(\p^n)$ (Proposition~\ref{Prop:SeqBoundedOrNot}).

In \cite{Popov},  V. L. Popov asks whether the set of unipotent elements of~$\Bir(\p^n)$ is closed, as it is the case in all linear algebraic groups. This also raises the question of knowing if the set of algebraic elements is in fact closed.

After giving some properties of the Zariski topology of~$\Bir(\p^n)$ in Section~\ref{Sec:Families}, we describe in Section~\ref{Sec:Examples} two families of birational maps that give the following result:
\begin{theorem}
For each~$n\ge 2$, there are two closed subsets~$U,S\subset \Bir(\p^n)$, canonically homeomorphic to~$\A^1$ and~$\A^1\times (\A^1\setminus \{0\})$ respectively, such that the following holds:
\begin{enumerate}[$(1)$]
\item
The set of algebraic elements of~$U$ is equal to the set of unipotent elements of~$U$, and corresponds to the elements~$t\in \A^1$ that belong to the subgroup of~$(\k,+)$ generated by~$1$.
\item
The set of algebraic elements of~$S$ is equal to the set of semi-simple elements of~$S$, and corresponds to the elements~$(a,\xi)\in \A^1\times (\A^1\setminus \{0\})$ such that~$a=\xi^k$ for some~$k\in \mathbb{Z}$.
\end{enumerate}
In particular, the set~$\Bir(\p^n)_{\mathrm{alg}}$ of algebraic elements of~$\Bir(\p^n)$ is not closed in~$\Bir(\p^n)$. Moreover, if~$\car(\k)=0$, the set of unipotent elements of~$\Bir(\p^n)$ is not closed.
\end{theorem}

Let us finish this introduction with some remarks:
\begin{enumerate}
\item
The set~$\Bir(\p^n)_{\mathrm{alg}}$ is a countable union of closed sets of~$\Bir(\p^n)$ (Proposition~\ref{Prop:Union}).
\item
We do not know if the set of unipotent elements of~$\Bir(\p^n)$ is closed in~$\Bir(\p^n)_{\mathrm{alg}}$ (although it is not closed in~$\Bir(\p^n))$.
\item
One can restrict ourselves to the subgroup~$\Aut(\A^n)\subset \Bir(\A^n)\simeq \Bir(\p^n)$. 
Over~$\k=\C$, it follows from \cite{FurterIt} that the set of algebraic elements of~$\Aut(\A^2_\C)$ is closed in~$\Aut(\A^2)$. The question is however open for~$\A^n$,~$n\ge 3$.
\end{enumerate}

The author thanks the referee for his careful reading and his corrections.

\section{A few properties of the Zariski topology of~$\Bir(\p^n)$}
\label{Sec:Families}
\subsection{Families of birational maps and the Zariski topology induced}
We recall the notion of families of birational maps, introduced by M.~Demazure in \cite{De} (see also \cite{Se}, \cite{BFannals}).

\begin{definition} \label{Defi:Family}
Let~$A,X$ be irreducible algebraic varieties, and let~$f$ be a~$A$-birational map of the~$A$-variety~$A\times X$, inducing an isomorphism~$U\to V$, where~$U,V$ are open subsets of~$A\times X$, whose projections on~$A$ are surjective.

The rational map~$f$ is given by~$(a,x)\dashmapsto (a,p_2(f(a,x)))$, where~$p_2$ is the second projection, and for each~$\k$-point~$a\in A$, the birational map~$x\dashmapsto p_2(f(a,x))$ corresponds to an element ~$f_a\in \Bir(X)$.
The map~$a\mapsto f_a$ represents a map from~$A$~$($more precisely from the~$A(k)$-points of~$A)$ to~$\Bir(X)$, and will be called a \emph{morphism} from~$A$ to~$\Bir(X)$.
\end{definition}
These notions yield the natural Zariski topology on~$\Bir(X)$, introduced by M.~Demazure \cite{De} and J.-P. Serre \cite{Se}:
\begin{definition}  \label{defi: Zariski topology}
A subset~$F\subseteq \Bir(X)$ is closed in the Zariski topology
if for any algebraic variety~$A$ and any morphism~$A\to \Bir(X)$ the preimage of~$F$ is closed.
\end{definition}
We can make the following simple observations:

\begin{lemma}\label{Lem:Continuous}Let~$X,Y$ be irreducible algebraic varieties, let~$\mu\colon X\dasharrow Y$ and~$\psi\colon X\dasharrow X$ be birational maps and let~$m\in \mathbb{Z}$ be some integer. The following maps are continuous
$$\begin{array}{ll}
\begin{array}{rrcl}
1) &\Bir(X)&\to& \Bir(X),\\
&\varphi & \mapsto & \psi\varphi\end{array} & \begin{array}{rrcl}
2)& \Bir(X)&\to& \Bir(X),\\
&\varphi & \mapsto & \varphi\psi\end{array}\\ \begin{array}{rrcl}
\vphantom{\Big)}3)& \Bir(X)&\to& \Bir(X),\\
&\varphi & \mapsto & \varphi^{m}\end{array} &\begin{array}{rrcl}
4)& \Bir(X)&\to& \Bir(X)\\
&\varphi & \mapsto & \varphi\psi\varphi^{-1},\end{array}\\
\begin{array}{rrcl}5)& \Bir(X)&\to& \Bir(Y)\\
&\varphi & \mapsto & \mu\varphi\mu^{-1},\end{array}\\
\end{array}$$
\end{lemma}
\begin{proof}Let ~$A$ be an irreducible algebraic variety.
If ~$f,g$ are two~$A$-birational maps~$f,g\colon A\times X\dasharrow A\times X$  inducing morphisms~$A\to \Bir(X)$, then~$f\circ g$ and~$f^{-1}$ are again~$A$-birational maps that induce morphisms~$A\to \Bir(X)$. This shows that the map~$\Bir(X)\to \Bir(X)$ given by~$\varphi\mapsto \varphi^m$ is continuous. Similarly,~$(\mathrm{id}\times \psi )\circ f$,~$f\circ (\mathrm{id}\times \psi )$ and~$f\circ (\mathrm{id}\times \psi )\circ f^{-1}$ are~$A$-birational maps that induce morphisms~$A\to \Bir(X)$, so the maps~$\Bir(X)\to \Bir(X)$ given by~$\varphi\mapsto \varphi\psi$, ~$\varphi\mapsto \psi\varphi$ and~$\varphi\psi\varphi^{-1}$ are continuous. The continuity of the last map is given in a similar way, by observing that~$(\mathrm{id}\times \mu^{-1} )\circ f\circ (\mathrm{id}\times \mu^{-1} )$ also yields a~$A$-birational map that induces a morphism~$A\to \Bir(X)$.
\end{proof}
\begin{corollary}\label{Cor:ClosureSubgroup}
Let~$\varphi\in \Bir(X)$. Denote by~$F$ the closure of~$\{\varphi^i\mid i\in \mathbb{Z}\}$ in~$\Bir(X)$. Then,~$F$ is a closed abelian  subgroup of~$\Bir(X)$.
\end{corollary}
\begin{proof}
The argument is the same as for algebraic groups or topological groups, and follows from Lemma~\ref{Lem:Continuous}, which gives the properties needed for the proof.
Let us recall how it works.

$1)$ For each~$j\in \mathbb{Z}$, the set~$\varphi^jF$ is a closed subset of~$\Bir(X)$ which contains~$\{\varphi^i\mid i\in \mathbb{Z}\}$, and contains thus~$F$. This implies that~$\varphi^jF=F$ for each~$j\in \mathbb{Z}$.

$2)$ Let us write~$M=\{\psi \in \Bir(X)\mid \psi F\subset F\}=\bigcap\limits_{f\in F} Ff^{-1}$. Since~$M$ is closed and contains~$\{\varphi^i\mid i\in \mathbb{Z}\}$,~$M$ contains~$F$. This shows that~$F$ is closed under composition.

$3)$ Similarly, the set~$I=\{\psi^{-1}\mid \psi \in F\}$ is closed in~$\Bir(X)$ and contains~$\{\varphi^i\mid i\in \mathbb{Z}\}$; hence it contains~$F$. The set~$F$ is then a subgroup of~$\Bir(X)$.

$4)$ It remains to see that~$F$ is abelian. 

We denote by~$C(\mu)=\{\psi\in \Bir(X)\mid \psi\mu=\mu\psi\}$ the centraliser of an element~$\mu\in \Bir(X)$. Note that~$C(\mu)$ is the preimage of the identity by the continuous map~$\Bir(X)\to \Bir(X)$ which sends~$\psi$ onto~$\psi\mu\psi^{-1}\mu^{-1}$. A point of~$\Bir(X)$ being closed by definition of the topology, this shows that~$C(\mu)$ is closed.

Because~$C(\varphi)$ is a closed subgroup of~$\Bir(X)$ which contains~$\{\varphi^i\mid i\in \mathbb{Z}\}$, it contains~$F$, so each element of~$F$ commutes with~$\varphi$. 

Finally,  we write~$S=\{\psi \in \Bir(X)\mid \psi f=f\psi \mbox{ for each } f\in F\}=\bigcap\limits_{f\in F} C(f)$, which is again closed, contains~$\{\varphi^i\mid i\in \mathbb{Z}\}$, and thus contains~$F$. This shows that~$F$ is abelian.
\end{proof}
\subsection{Reminders of results of \cite{BFannals}}
Recall the following natural construction associated to~$\Bir(\p^n)$ (which is \cite[Definition 2.3]{BFannals}):
\begin{definition}\label{DefWHG}
Let~$d$ be a positive integer.
\begin{enumerate}
\item
We define~$W_d$ to be the set of equivalence classes of non-zero ~$(n+1)$-uples~$(h_0,\dots,h_n)$
of homogeneous polynomials~$h_i\in \k[x_0,\dots,x_n]$ of degree~$d$,
where~$(h_0,\dots,h_n)$ is equivalent to~$(\lambda h_0,\dots,\lambda h_n)$ for any~$\lambda\in \k^{*}$.
The equivalence class of~$(h_0,\dots,h_n)$ will be denoted by~$(h_0:\dots:h_n)$.
\item
We define~$H_d\subseteq W_d$ to be the set of elements~$h=(h_0:\dots:h_n)\in W_d$
such that the rational map
$\psi_h\colon \p^n\dasharrow \p^n$ given by~$(x_0:\dots:x_n)\dashmapsto
(h_0(x_0,\dots,x_n):\dots:h_n(x_0,\dots,x_n))$ is birational.
We denote by~$\pi_d$ the map~$H_d\to \Bir(\p^n_\k)$ which sends~$h$ onto~$\psi_h$.
\end{enumerate}
\end{definition}
It follows from the construction that~$W_d$ is a projective space and that~$\pi_d(H_d)=\Bir(\p^n)_{\le d}$. Moreover, we have the following properties:
\begin{proposition}\label{Prop:StructureHd}\cite[Lemma 2.4, Corollaries 2.7 and 2.9]{BFannals}\ 

$1)$ The set~$H_d\subset W_d$ is locally closed, and is thus an algebraic variety.

$2)$ The map~$\pi_d\colon H_d\to \Bir(\p^n)$ is a morphism. It yields a map~$H_d\to \Bir(\p^n)_{\le d}$ which is surjective, closed and continuous. In particular, it is a topological quotient map.

$3)$ A subset~$F\subset \Bir(\p^n)$ is closed if and only if~$(\pi_d)^{-1}(F)$ is closed in~$H_d$ for each~$d$.
\end{proposition}

We also have the following description of algebraic subgroups of~$\Bir(\p^n)$:

\begin{proposition}\label{Prop:EquivAlgSub}\cite[Corollary 2.18 and Lemma 2.19]{BFannals}\ 

A subgroup of~$\Bir(\p^n)$ is an algebraic subgroup if and only if it is closed and of bounded degree. 
\end{proposition}
\subsection{Algebraicity and boundedness of the degree sequence}
\begin{proposition}\label{Prop:SeqBoundedOrNot}
Let~$\varphi\in \Bir(\p^n)$. 

$1)$ If the sequence~$\{\deg(\varphi^m)\}_{m\in \mathbb{N}}$ is bounded, then~$\overline{\{\varphi^i\mid i\in \mathbb{Z}\}}$ is a commutative algebraic subgroup of~$\Bir(\p^n)$.

$2)$ If the sequence~$\{\deg(\varphi^m)\}_{m\in \mathbb{N}}$ is unbounded, then~$\varphi$ is not contained in any algebraic subgroup of~$\Bir(\p^n)$.
\end{proposition}
\begin{proof}
Proposition~\ref{Prop:EquivAlgSub} directly yields~$2)$. Let us prove~$1)$. 

We suppose then that~$\{\deg(\varphi^m)\}_{m\in \mathbb{N}}$ is bounded. Because~$\deg(\varphi^{-m})\le (\deg(\varphi^m))^{n-1}$ for each~$m$ (\cite[Theorem 1.5, page 292]{BCW}), the set 
$\{\varphi^i\}_{i\in \mathbb{Z}}$ is contained in~$\Bir(\p^n)_{\le d}$ for some~$d$. The closure~$F$ of ~$\{\varphi^i\mid i\in \mathbb{Z}\}$ in~$\Bir(\p^n)$ is then again contained in~$\Bir(\p^n)_{\le d}$ . By Corollary~\ref{Cor:ClosureSubgroup},~$F$ is a commutative subgroup of~$\Bir(\p^n)$ and is then a commutative algebraic subgroup of~$\Bir(\p^n)$ (Proposition~\ref{Prop:EquivAlgSub}).
\end{proof}
\begin{corollary}\label{Cor:AlgSequenceBounded}
Let~$\varphi\in \Bir(\p^n)$. The following are equivalent:

$1)$ The element~$\varphi$ is algebraic.

$2)$ The sequence~$\{\deg(\varphi^m)\}_{m\in \mathbb{N}}$ is bounded.
\end{corollary}
\begin{proof}
Directly follows from Proposition~\ref{Prop:SeqBoundedOrNot}.
\end{proof}

\begin{lemma}
The Zariski topology  of~$\Bir(\p^n)_{\le d}$ is noetherian, i.e.~every decreasing sequence of closed subsets is eventually stationary.

This is not the case for~$\Bir(\p^n)$.
\end{lemma}
\begin{proof}By  Proposition~\ref{Prop:StructureHd}, we have a map~$\pi_d\colon H_d\to \Bir(\p^n)_{\le d}$, which is surjective, continuous and closed. The topology of~$H_d$ being noetherian (it is an algebraic variety), the same holds for~$\Bir(\p^n)_{\le d}$.

The fact that the topology of~$\Bir(\p^n)$ is not noetherian has already being observed in \cite{PanRit}. It can be shown by taking a sequence~$\{\varphi_i\}_{i\in \mathbb{N}}$ of maps~$\varphi_i$ of degree~$i$. Then~$F_i=\{\varphi_j\mid j\ge i\}$ is closed in~$\Bir(\p^n)$ for each~$i$ (follows from Proposition~\ref{Prop:StructureHd})  but the sequence~$F_1\supset F_2\supset F_3\supset \dots$ is not stationary.
\end{proof}

\begin{proposition}\label{Prop:Union}
For each integers~$k,d\in \mathbb{N}$ let us write~$$\begin{array}{rcl}
\Bir(\p^n)_{k,d}&=&\{f\in \Bir(\p^n)\mid \deg(f^k)\le d\}\\
\Bir(\p^n)_{\infty,d}&=&\{f\in \Bir(\p^n)\mid \deg(f^i)\le d\mbox{ for all }i\in \mathbb{N}\}.\end{array}$$ Then, the following hold:
\begin{enumerate}[$(1)$]
\item
The set~$\Bir(\p^n)_{k,d}$ is closed in~$\Bir(\p^n)$.
\item
The set~$\Bir(\p^n)_{\infty,d}=\bigcap_{i\in \mathbb{N}} \Bir(\p^n)_{i,d}$ is closed in~$\Bir(\p^n)$.
\item
The set~$\Bir(\p^n)_{alg}$ of all algebraic elements of~$\Bir(\p^n)$ is equal to the union of all~$\Bir(\p^n)_{\infty,d}$.
\end{enumerate}
\end{proposition}
\begin{proof}
By Proposition~\ref{Prop:StructureHd}, the set~$\Bir(\p^n)_{\le d}$ is closed in~$\Bir(\p^n)$ for each~$d$. The map~$\Bir(\p^n)\to \Bir(\p^n)$ which sends~$\varphi$ onto~$\varphi^k$ being continuous (Lemma~\ref{Lem:Continuous}), this directly shows that 
$\Bir(\p^n)_{k,d}$ is closed in~$\Bir(\p^n)$. This yields~$(1)$, which  implies~$(2)$.

Corollary~\ref{Cor:AlgSequenceBounded} yields the equality~$\Bir(\p^n)_{alg}=\bigcup_{d\in \mathbb{N}} \Bir(\p^n)_{\infty,d}$, which corresponds to~$(3)$.
\end{proof}

\section{Two explicit families}
\label{Sec:Examples}
\subsection{A unipotent example}
\begin{example}\label{TheExample}
For~$n\ge 2$, let ~$\rho\colon \A^1\to \Bir(\p^n)$ be the morphism given by
$$\begin{array}{cccc}
 \A^1\times\p^n&\dasharrow & \A^1\times\p^n\\
(a,[x_0:\dots:x_n])&\dashmapsto& (a,[x_0x_1:x_1(x_1+x_0):x_2(x_1+ax_0):x_3x_1:\dots:x_nx_1])
\end{array}$$
which corresponds on the affine open subset where~$x_0=1$ to the family of birational maps given by
$$(x_1,\dots,x_n)\dashmapsto (x_1+1,x_2\cdot \frac{x_1+a}{x_1},x_3,\dots,x_n).$$
\end{example}
\begin{lemma}The map~$\rho\colon \A^1\to \Bir(\p^n)$ of Example~$\ref{TheExample}$ is a topological embedding.\label{Lem:Topemb1}
\end{lemma}
\begin{proof}The fact that~$\rho$ is injective can be directly checked on the formula given above.
We then consider the closed embedding~$\hat{\rho}\colon\p^1\to W_2$ that sends~$[\mu:\lambda]\in \p^1$ to 
$$[\mu x_0x_1:\mu x_1(x_1+x_0): \mu x_2x_1+\lambda x_2x_0: \mu x_3x_1:\dots: \mu x_nx_1].$$
When~$\mu=0$, this does not give a birational map of~$\p^n$, so ~$\hat{\rho}([0:1])\notin H_2$. However, we have~$\pi_2(\hat{\rho}([1:t]))=\rho(t)$ for each~$t\in \A^1$, so the restriction to~$\A^1$ yields a closed embedding~$\A^1\to H_2$. It remains to show that the restriction of~$\pi_2$ to~$\hat{\rho}(\p^1\setminus [0:1])$ is an homeomorphism, which is given by Proposition~\ref{Prop:StructureHd}.
\end{proof}
\begin{proposition}
The morphism~$\rho\colon \A^1\to \Bir(\p^n)$ of Example~$\ref{TheExample}$ has the following properties:
\begin{enumerate}[$(1)$]
\item
For~$t\in \k$, the following conditions are equivalent:
\begin{enumerate}[$(a)$]
\item
$\rho(t)$ is algebraic;
\item
$\rho(t)$ is unipotent;
\item
$\rho(t)$ is conjugate to~$\rho(0)\colon (x_1,\dots,x_n)\mapsto (x_1+1,x_2,\dots,x_n)$;
\item
$t$ belongs to the subgroup of~$(\k,+)$ generated by~$1$.
\end{enumerate}
\item
The pull-back by~$\rho$ of the set of algebraic elements is not closed if~$\car(\k)=0$.
\end{enumerate}
\end{proposition}
\begin{proof}$(1)$
Proceeding by induction, the iterates of~$\rho(a)$ send~$(x_1,\dots,x_n)$ onto:
$$\begin{array}{ll}
\displaystyle\rho(a)\colon &(x_1+1,x_2\cdot \displaystyle\frac{x_1+a}{x_1},x_3,\dots,x_n),\\
\rho(a)^2\colon &   (x_1+2,x_2\cdot \displaystyle\frac{(x_1+a)(x_1+a+1)}{x_1(x_1+1)},x_3,\dots,x_n),\\
\rho(a)^m\colon &  (x_1+m,x_2\cdot \displaystyle\frac{(x_1+a)(x_1+a+1)\cdots (x_1+a+m-1)}{x_1(x_1+1)\cdots (x_1+m-1)},x_3,\dots,x_n).
\end{array}$$
Then, the second coordinate of~$\rho(a)^m(x_1,\dots,x_n)$ is
$$\frac{\prod_{i=0}^{m-1} (x_1+a+i)}{\prod_{i=0}^{m-1} (x_1+i)}.$$
If~$a$ does not belong to the subgroup of~$(\k,+)$ generated by~$1$, then the denominator and numerators have no common factor, for each~$m\in \mathbb{N}$, so the degree growth of~$\rho(a)^m$ is linear, which implies that~$\rho(a)$ is not algebraic.

If~$a$ belongs to the subgroup of~$(\k,+)$ generated by~$1$, it is equal to~$k\in \mathbb{Z}$, and the degree of~$\{\rho(a)^m\}_{m\in \mathbb{N}}$ is bounded by~$\lvert k\rvert+1$, so~$\rho(a)$ is algebraic. We can moreover see that~$\rho(a)$ is unipotent. Indeed,~$\rho(a)$ is conjugate to 
$$\rho(0)=(x_1,\dots,x_n)\dashmapsto (x_1+1,x_2,x_3,\dots,x_n)$$
by 
$$(x_1,\dots,x_n)\dashmapsto (x_1,\frac{x_2}{x_1(x_1+1)\dots(x_1+a-1)},x_3,\dots,x_n)$$
if~$a=k>0$
or by 
$$(x_1,\dots,x_n)\dashmapsto (x_1,x_2\cdot x_1(x_1-1)\dots(x_1+a),x_3,\dots,x_n)$$
if~$a=k<0$.

Assertion~$(2)$ follows directly from~$(1)$ and the fact that the subgroup of~$(\k,+)$ generated by~$1$ is closed if and only if~$\car(\k)\not=0$.
\end{proof}

\subsection{A semi-simple example}
\begin{example}\label{TheSecondExample}
For~$n\ge 2$, let~$\rho\colon \A^1\times \left(\A^1\setminus\{0\}\right)\to \Bir(\p^n)$ be the morphism given by
$$\begin{array}{l}
\rho(a,\xi)([x_0:x_1:\dots:x_n])=\\ \ \ \ [x_0(x_1+x_0):\xi x_1(x_1+x_0):x_2(x_1+ax_0):x_3(x_1+x_0):\dots:x_n(x_1+x_0)]
\end{array}$$
which corresponds on the affine open subset where~$x_0=1$ to the family of birational maps
$$(x_1,\dots,x_n)\dashmapsto \left(\xi x_1,x_2\cdot \frac{x_1+a}{x_1+1},x_3,\dots,x_n\right).$$
\end{example}

\begin{lemma}The map~$\rho\colon \A^1\to \Bir(\p^n)$ of Example~$\ref{TheSecondExample}$ is a topological embedding.
\end{lemma}
\begin{proof}The proof is similar to the one of Lemma~\ref{Lem:Topemb1}. The fact that~$\rho$ is injective can be directly checked on the formula given above.
We then consider the closed embedding~$\hat{\rho}\colon\p^2\to W_2$ that sends~$[\mu:\eta:\lambda]\in \p^2$ to 
$$[\mu x_0(x_1+x_0):\lambda x_1(x_1+x_0): x_2(\mu x_1+ \eta x_0):\mu x_3(x_1+x_0):\dots:\mu x_n(x_1+x_0)].$$
These elements correspond to birational maps if and only if~$\mu\lambda\not=0$. Hence, we have a closed embedding~$\A^1\times (\A^1\setminus\{0\})\to H_2$ that sends~$(a,\xi)$ onto~$\hat{\rho}([1:a:\xi])$. Moreover, we have~$\pi_2(\hat{\rho}([1:a:\xi]))=\rho(a,\xi)$. The fact that the restriction of~$\pi_2$ to the image is a homeomorphism is then given by Proposition~\ref{Prop:StructureHd}.
\end{proof}
\begin{proposition}
The morphism~$\rho\colon \A^1\to \Bir(\p^n)$ of Example~$\ref{TheSecondExample}$ has the following properties:
\begin{enumerate}[$(1)$]
\item
For~$(a,\xi)\in \A^1\times (\A^1\setminus \{0\})$, the following conditions are equivalent:
\begin{enumerate}[$(a)$]
\item
$\rho(a,\xi)$ is algebraic;
\item
$\rho(a,\xi)$ is semi-simple;
\item
$\rho(a,\xi)$ is conjugate to~$\rho(1,\xi)\colon (x_1,\dots,x_n)\mapsto (\xi x_1,x_2,\dots,x_n)$;
\item
There exists~$k\in \mathbb{Z}$ such that~$a=\xi^k$.
\end{enumerate}
\item
The pull-back by~$\rho$ of the set of algebraic elements is not closed.
\end{enumerate}
\end{proposition}

\begin{proof}$(1)$
Proceeding by induction, the iterates of~$\rho(a,\xi)$ send~$(x_1,\dots,x_n)$ onto:
$$\begin{array}{ll}
\displaystyle\rho(a,\xi)\colon &\left(\xi x_1,x_2\cdot \displaystyle\frac{x_1+a}{x_1+1},x_3,\dots,x_n\right),\vspace{0.2cm}\\ 
\rho(a,\xi)^2\colon &   \left(\xi^2x_1,x_2\cdot \displaystyle\frac{(x_1+a)(\xi x_1+a)}{(x_1+1)(\xi x_1+1)},x_3,\dots,x_n\right),\vspace{0.2cm}\\
\rho(a,\xi)^m\colon &  \left(\xi^m x_1,x_2\cdot \displaystyle\frac{(x_1+a)(\xi x_1+a)\cdots (x_1\xi^{m-1}+a)}{(x_1+1)(x_1\xi+1)\cdots (x_1\xi^{m-1}+1)},x_3,\dots,x_n\right).
\end{array}$$
Then, the second coordinate of~$\rho(a,\xi)^m(x_1,\dots,x_n)$ is
$$\frac{\prod_{i=0}^{m-1} (\xi^i x_1+a)}{\prod_{i=0}^{m-1} (\xi^i x_1+1)}.$$

If~$a$ does not belong to the subgroup of~$(\k,\cdot)$ generated by~$\xi$, then the denominator and numerators have no common factor, for each~$m\in \mathbb{N}$, so the degree growth of~$\rho(a,\xi)^m$ is linear, which implies that~$\rho(a,\xi)$ is not algebraic.

If~$a$ belongs to the subgroup of~$(\k,\cdot)$ generated by~$\xi$, it is equal to~$a=\xi^k$ for some~$k\in \mathbb{Z}$, and the degree of~$\{\rho(a,\xi)^m\}_{m\in \mathbb{N}}$ is bounded by~$\lvert k\rvert+1$, so~$\rho(a,\xi)$ is algebraic. We can moreover see that~$\rho(a,\xi)$ is semi-simple. Indeed,~$\rho(a,\xi)$ is conjugate to 
$$\rho(1,\xi)=(x_1,\dots,x_n)\dashmapsto (\xi x_1,x_2,x_3,\dots,x_n)$$
by 
$$(x_1,\dots,x_n)\dashmapsto \left(x_1,x_2\prod_{i=1}^k (\xi^{-i}+\frac{1}{x_1}),x_3,\dots,x_n\right)$$
if~$k>0$
or by
$$(x_1,\dots,x_n)\dashmapsto \left(x_1,\frac{x_2}{\prod_{i=0}^{-k-1} (\xi^{i}+\frac{1}{x_1})},x_3,\dots,x_n\right)$$
if~$k<0$.

Assertion~$(2)$ follows  from~$(1)$ and the fact that the subset of~$\A^1\times (\A^1\setminus \{0\})$ that consists of elements~$(a,\xi)$ such that~$a=\xi^k$ for some~$k\in \mathbb{Z}$ is not closed.
\end{proof}

\end{document}